\documentclass[11pt,a4paper]{amsart}
\usepackage[utf8]{inputenc}
\usepackage[T1]{fontenc}
\usepackage[english]{babel}

\usepackage[margin=3.5cm]{geometry}

\usepackage{tikz}

\usepackage{amssymb}
\usepackage{empheq}
\usepackage{esint}
\usepackage{xcolor}
\usepackage{hyperref}
\hypersetup{colorlinks = true, linkbordercolor = red, citecolor = blue}

\usepackage[alphabetic, msc-links]{amsrefs}
\renewcommand{\eprint}[1]{\href{http://arxiv.org/abs/#1}{\texttt{arXiv:#1}}}

\usepackage{appendix}

\numberwithin{equation}{section}

\newtheoremstyle{newthm}
{.5em} 
{.5em} 
{\it}       
{}          
{\bf}       
{.}         
{.5em}      
{}          %
\theoremstyle{newthm}
\newtheorem{thm}{Theorem}[section]
\newtheorem{lem}[thm]{Lemma}
\newtheorem{prp}[thm]{Proposition}

\newtheoremstyle{newdef}
{.5em} 
{.5em} 
{}          
{}          
{\bf}       
{.}         
{.5em}      
{}          %
\theoremstyle{newdef}
\newtheorem{defn}[thm]{Definition}
\newtheorem{rem}[thm]{Remark}

\newcommand{\R}{\mathbf{R}}
\renewcommand{\S}{\mathbf{S}}
\newcommand{\B}{\mathbf{B}}
\newcommand{\RP}{\mathbf{RP}}

\DeclareMathOperator{\im}{im}
\DeclareMathOperator{\cat}{cat}

\newcommand{\g}{\gamma}
\renewcommand{\d}{\delta}
\newcommand{\e}{\varepsilon}
\newcommand{\z}{\zeta}
\renewcommand{\th}{\theta}
\renewcommand{\k}{\kappa}
\newcommand{\n}{\nu}
\newcommand{\x}{\xi}
\newcommand{\s}{\sigma}
\renewcommand{\t}{\tau}
\newcommand{\f}{\phi}
\newcommand{\vp}{\varphi}

\newcommand{\G}{\Gamma}
\newcommand{\Th}{\Theta}
\newcommand{\F}{\Phi}
\renewcommand{\O}{\Omega}

\newcommand{\p}{\partial}

\begin{document}

\title{Minimal networks on balls and spheres for almost standard metrics}

\author{Luciano Sciaraffia}
\address{Dipartimento di Matematica, Università di Pisa, Largo Bruno Pontecorvo 5, 56127 Pisa, Italy.}
\email{luciano.sciaraffia@phd.unipi.it}



\begin{abstract}
We study the existence of minimal networks in the unit sphere $\S^d$ and the unit ball $\B^d$ of $\R^d$ endowed with Riemannian metrics close to the standard ones.
We employ a finite-dimensional reduction method, modelled on the configuration of $\th$-networks in $\S^d$ and triods in $\B^d$, jointly with the Lusternik--Schnirelmann category.
\end{abstract}

\maketitle

\tableofcontents


\section{Introduction and main results}

In this paper, we are devoted to the study of \textit{minimal networks}, often referred to as \textit{geodesic nets} in a Riemannian context, which can be viewed as singular generalisations of geodesics.
These objects naturally arise as solutions to the classical \textit{Steiner problem:} given a finite set of points, find a connected set minimising length and spanning said points.
They can also be found as solutions to minimising problems for the length in one-dimensional homology classes, as discussed, for example, in \cite{Morgan-Currents1989}.
As a consequence of the first variation of length, minimal networks are characterised by being composed as unions of geodesic segments, intersecting only at their endpoints in junctions of multiplicity at most three, and the angles formed at each junction between two consecutive segments are greater or equal than $2\pi/3$.
Therefore, at a triple junction equal angles of $2\pi/3$ are formed.
For a comprehensive treatment of the subject, we refer to \cite{IvanovTuzhilin-MinNets1994}.
For further generalisations and variations, consult \cite{IvanovTuzhilin-MinNetsRev2016}.

One of the most renowned problems in differential geometry is to find and count closed geodesics on a given compact Riemannian manifold, which represent the simplest instances of closed minimal networks.
Although this problem remains unsolved in full generality, we highlight the two-dimensional case.
In the case of compact surfaces with positive genus, it is always possible to find a closed geodesic representative within a given nontrivial homotopy class by directly minimising the length functional.
However, in the simply connected case, the situation becomes more intricate, as any direct minimisation leads to degenerate curves.
The first result in this direction was due to Birkhoff \cite{Birkhoff-DynSys2degs1917}, who introduced the min-max technique and established the existence of a closed geodesic on an arbitrary Riemannian 2-sphere.
Regarding multiplicity, Lusternik--Schnirelmann \cite{LusternikSchnirelmann-3geos1929} claimed that any Riemannian 2-sphere possesses at least three embedded closed geodesics.
Their result remained contentious due to a gap in the proof, but nowadays the gap is considered filled due, for instance, to the work of Grayson \cite{Grayson-ShortCurves1989}.
We also mention an approach, first suggested by Poincaré \cite{Poincare-GeosConvex1905}, to find closed geodesics in a convex sphere.
He stated that a closed geodesic could be found minimising the length among all curves which divide the sphere into two pieces of equal total curvature.
This plan was brought to fruition by Croke \cite{Croke-ShortestGeo1982}, and a simplified proof was later provided by Hass--Morgan \cite{HassMorgan-Bubbles1996}.
It is also noteworthy to mention the works of Berger--Bombieri \cite{BergerBombieri-PoincareIso1981} and Klingenberg \cite{Klingenberg-PoincareGeoConvex2004} related to this approach.

Much like the case of geodesics, finding minimal networks in surfaces with positive genus can be achieved by the direct method of the calculus of variations.
In their work \cite{MartelliNovagaPludaRiolo-Spines2017}, Martelli--Novaga--Pluda--Riolo extensively explore the existence and classification of \textit{minimal spines}, that is, networks whose complement can be retracted to a disc.
Specifically, they establish that on every orientable surface of positive genus there exists a spine of minimal length, which in the case of a torus must be a $\th$-\textit{network}, that is, a closed network consisting of three geodesics and two triple junctions resembling the Greek letter theta.
Furthermore, they provide a comprehensive classification of minimal spines on flat tori, including an explicit counting function defined in their moduli space. 
A partial result in hyperbolic surfaces was also obtained, which states that the number of minimal spines with uniformly bounded length is finite.
It is also worth mentioning earlier work by Ivanonv--Ptitsyna--Tuzhilin \cite{IvanovPtitsynaTuzhilin-MinNetsTori1992}, where all the minimal networks in flat tori are classified up to ambient isotopies, accompanied by a specific criterion for determining whether a given type of minimal network can be embedded on a given flat torus.

The case of a sphere, in analogy with closed geodesics, has proven to be more challenging, as the previous direct minimisation methods again only yield degenerate curves.
For a standard round sphere, however, Heppes \cite{Heppes-NetsSphere1964} exploited the constant curvature and the Gauss--Bonnet Theorem to show that there are exactly nine types of minimal networks other than closed geodesics, which are also unique up to isometries.
These networks partition the sphere into three to ten, or twelve regions, each bounded by at most five arcs.
When it comes to non-standard Riemannian metrics, the situation is considerably more intricate, and limited knowledge is available.
Morgan \cite{Morgan-SoapBubbles1994} and Hass--Morgan \cite{HassMorgan-Bubbles1996} studied the \textit{isoperimetric problem} for clusters, which amounts to finding the sets that minimise length and which enclose prescribed weighted areas, where the weight is a positive function defined in the manifold.
Then, inspired by Poincaré's original idea mentioned earlier, and building upon the techniques developed in their previous work, in \cite{HassMorgan-NetsSphere1996} they used the total Gaussian curvature as the weighted area for clusters dividing a given convex Riemannian sphere into three different regions.
In this setting, they showed the existence of at least one minimal network, which by a straightforward combinatorial argument must be either a $\th$-\textit{network}, \textit{eyeglasses}, or a \textit{figure 8}, with the later serving as a degenerate limit between the former two (see Figure \ref{fig:net-types-S}).
Moreover, as a corollary, they showed that metrics close to the standard one in $C^2$ must support a minimal $\th$-network, just as the round metric of constant curvature does.
It is remarkable that, to this day, this remains one of the few existing results of its kind.
In particular, it is not known whether every Riemannian metric on a sphere with positive curvature admits a minimal $\th$-network, as far as we are aware.

\begin{figure}
\centering

\begin{tikzpicture}[x=1cm,y=1cm,scale=.8]


\begin{scope}[xscale=-1]
\draw[line width=1pt, smooth, shift={(5 cm, 0 cm)}] plot[samples=200,domain=-2:0] (\x,{.54*(-\x+1)*sqrt(2+\x)});
\draw[line width=1pt, smooth, shift={(5 cm, 0 cm)}] plot[samples=200,domain=-2:0] (\x,{-.54*(-\x+1)*sqrt(2+\x)});
\end{scope}

\draw[line width=1pt, smooth, shift={(-5 cm, 0 cm)}] (0,.77) -- (0,-.77);

\draw[line width=1pt, smooth, shift={(-5 cm, 0 cm)}] plot[samples=200,domain=-2:0] (\x,{.54*(-\x+1)*sqrt(2+\x)});
\draw[line width=1pt, smooth, shift={(-5 cm, 0 cm)}] plot[samples=200,domain=-2:0] (\x,{-.54*(-\x+1)*sqrt(2+\x)});


\begin{scope}[xscale=-1]
\draw[line width=1pt, smooth, shift={(0 cm, 0 cm)}] plot[samples=200,domain=-2:-.5] (\x,{(-\x-.5)*sqrt(4+2*\x)});
\draw[line width=1pt, smooth, shift={(0 cm, 0 cm)}] plot[samples=200,domain=-2:-.5] (\x,{-(-\x-.5)*sqrt(4+2*\x)});
\end{scope}

\draw[line width=1pt, smooth, shift={(0 cm, 0 cm)}] (-.5,0) -- (.5,0);

\draw[line width=1pt, smooth, shift={(0 cm, 0 cm)}] plot[samples=200,domain=-2:-.5] (\x,{(-\x-.5)*sqrt(4+2*\x)});
\draw[line width=1pt, smooth, shift={(0 cm, 0 cm)}] plot[samples=200,domain=-2:-.5] (\x,{-(-\x-.5)*sqrt(4+2*\x)});

\draw[line width=1pt, smooth, shift={(5 cm, 0 cm)}] plot[samples=200,domain=-45:45] (\x: {2*sqrt(cos(2*\x))});
\draw[line width=1pt, smooth, shift={(5 cm, 0 cm)}] plot[samples=200,domain=-45:45] (\x: {-2*sqrt(cos(2*\x))});

\end{tikzpicture}

\caption{The three possible minimal networks which divide a sphere into three regions: the $\th$-\textit{network,} the \textit{eyeglasses,} and the \textit{figure 8.}}
\label{fig:net-types-S}

\end{figure}
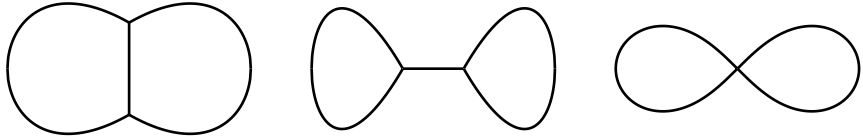

The first of our main results is a generalisation of \cite{HassMorgan-NetsSphere1996}*{Corollary 1} to higher dimensions, which also includes a lower bound on the number of such minimal networks.

\begin{thm}\label{thm:mainS}
Let $d \geq 2$ be an integer, and let $g_0$ be the standard round metric on the sphere $\S^d$.
Then for any Riemannian metric $g$ on $\S^d$ sufficiently close to $g_0$ in the $C^2$ topology there exist at least four minimal $\th$-networks on $(\S^d,g)$, when $d \neq 3,4$, and at least five, when $d = 3,4$.
\end{thm}

An analogous problem can also be explored when the manifold has boundary.
In this context, Freire \cite{Freire-SteinerConvex2011} delved into the existence of minimal networks within bounded, strictly convex planar sets which satisfy a Neumann condition at the boundary.
Freire's investigation unveiled three distinct types of possible configurations: a \textit{triod}, a \textit{double triod}, and a \textit{hexagonal cell}, as illustrated in Figure \ref{fig:net-types-B}.
Surprisingly, and contrary to previous beliefs, he showed that minimal triods might not always exist, constructing explicit counterexamples.
Nevertheless, existence of minimal triods can be guaranteed under some pinching conditions on the boundary, under which at least two different solutions exist \cite{Freire-SteinerConvex2011}*{Proposition 3.2}.
Our next theorem gives a generalisation of this fact to higher dimensions.

\begin{figure}
\centering

\begin{tikzpicture}[x=1cm,y=1cm,scale=1.4]

\draw[line width=.5pt, smooth, shift={(-3 cm, 0 cm)}] (0,0) circle (1cm);

\draw[line width=1pt, smooth, shift={(-3 cm, 0 cm)}] (0,0) -- (0,1);
\draw[line width=1pt, smooth, shift={(-3 cm, 0 cm)}] (0,0) -- (-.87,-.5);
\draw[line width=1pt, smooth, shift={(-3 cm, 0 cm)}] (0,0) -- (.87,-.5);

\draw[line width=.5pt, smooth, shift={(0 cm, 0 cm)}] (0,0) ellipse (1.3cm and 1cm);

\draw[line width=1pt, smooth, shift={(0 cm, 0 cm)}] (-.78,.8) -- (-.32,0);
\draw[line width=1pt, smooth, shift={(0 cm, 0 cm)}] (-.78,-.8) -- (-.32,0);
\draw[line width=1pt, smooth, shift={(0 cm, 0 cm)}] (.78,.8) -- (.32,0);
\draw[line width=1pt, smooth, shift={(0 cm, 0 cm)}] (.78,-.8) -- (.32,0);
\draw[line width=1pt, smooth, shift={(0 cm, 0 cm)}] (-.32,0) -- (.32,0);

\draw[line width=.5pt, smooth, shift={(3 cm, 0 cm)}] (0,0) circle (1cm);

\draw[line width=1pt, smooth, shift={(3 cm, 0 cm)}] (.7,0) -- (.35,.61);
\draw[line width=1pt, smooth, shift={(3 cm, 0 cm)}] (.7,0) -- (.35,-.61);
\draw[line width=1pt, smooth, shift={(3 cm, 0 cm)}] (.35,.61) -- (-.35,.61);
\draw[line width=1pt, smooth, shift={(3 cm, 0 cm)}] (.35,-.61) -- (-.35,-.61);
\draw[line width=1pt, smooth, shift={(3 cm, 0 cm)}] (-.35,.61) -- (-.7,0);
\draw[line width=1pt, smooth, shift={(3 cm, 0 cm)}] (-.35,-.61) -- (-.7,0);

\draw[line width=1pt, smooth, shift={(3 cm, 0 cm)}] (.7,0) -- (1,0);
\draw[line width=1pt, smooth, shift={(3 cm, 0 cm)}] (.35,.61) -- (.5,.87);
\draw[line width=1pt, smooth, shift={(3 cm, 0 cm)}] (-.35,.61) -- (-.5,.87);
\draw[line width=1pt, smooth, shift={(3 cm, 0 cm)}] (-.7,0) -- (-1,0);
\draw[line width=1pt, smooth, shift={(3 cm, 0 cm)}] (-.35,-.61) -- (-.5,-.87);
\draw[line width=1pt, smooth, shift={(3 cm, 0 cm)}] (.35,-.61) -- (.5,-.87);

\end{tikzpicture}

\caption{The three possible minimal networks in a strictly convex domain: the \textit{triod,} the \textit{double triod,} and the \textit{hexagonal cell.}}
\label{fig:net-types-B}

\end{figure}
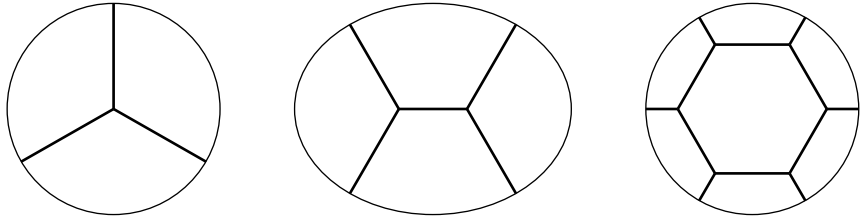

\begin{thm}\label{thm:mainB}
Let $d \geq 2$ be an integer, and let $\d_0$ be the standard Euclidean metric on the ball $\B^d$.
Then for any Riemannian metric $g$ on $\B^d$ sufficiently close to $\d_0$ in the $C^2$ topology there exist at least two, when $d = 2$, at least four, when $d = 3$, and at least three, when $d \geq 4$, minimal triods on $(\B^d,g)$.
\end{thm}

\begin{rem}
The multiplicities stated in both Theorems \ref{thm:mainS} and \ref{thm:mainB} are not optimal, as our method does not yield the exact number.
They are given by a lower bound on the Lusternik--Schnirelmann category (see Definition \ref{def:LS}) of suitable compact manifolds, whose exact values we are unable to compute for dimensions $d \geq 3$ (see Remarks \ref{rem:S3action} and \ref{rem:S3action-S}).
It could be possible to improve this bound, and in fact we propose a conjecture: the number of minimal $\th$-networks in the case of the sphere, and minimal triods in the case of the ball, as presented in the aforementioned theorems, tends to infinity as the dimension $d \to \infty$.
\end{rem}

We now outline how we prove Theorems \ref{thm:mainS} and \ref{thm:mainB}. 
In each case, we first identify a suitable finite-dimensional manifold parametrising the set of minimal networks under consideration for the standard metrics.
We show that these manifolds are \textit{non-degenerate in the sense of Bott} or, in other words, at each point of the manifold the second variation for the length functional has a kernel corresponding to the tangent space of the manifold at that point.
We then slightly perturb the Riemannian metrics $g$ to still find a set of approximately minimal networks parametrised by the minimal manifold, and define a \textit{reduced length functional} over said manifold as the length of the constructed networks.
For the case of the sphere, these networks consist of curves with “constant curvature,” in the sense that the curvature vector $k$ of each edge $\g$ has constant size $|k|_g$, which might not be zero, and the curves themselves might not meet the equiangularity condition at one of the junctions of the network.
For the case of the ball, we construct them as geodesic segments which might not be orthogonal, with respect to the perturbed metric, to the boundary $\p\B^d$.
The subsequent step involves demonstrating that these constraints are natural: we establish that the critical points of the reduced length functional are also critical points for the length functional in the space of networks.
We point out that is in this step that we need the $C^2$-closeness of the perturbed metric $g$, since we need to approximate the variational fields of the networks generated by tangent vectors to the minimal manifolds by means of an ODE that involves second derivatives of $g$. 
Finally, we take advantage of the non-degeneracy to perform a Lyapunov--Schmidt type reduction to localise the critical points of the reduced functional.
A direct application of the Lusternik--Schnirelman Theorem \ref{thm:LS} then yields the multiplicity of solutions.

To conclude this introductory discussion, let us highlight the main challenges one finds when trying to construct minimal or stationary networks.
The first one is that existence, as we see in the results of Freire, is not guaranteed even under some reasonable hypotheses like convexity.
Another major difficulty is that the standard techniques used for finding closed geodesics, such as Morse theory and curvature flows, may not lead to a nontrivial minimal network but rather to a closed geodesic.
Indeed, this is the case in, for example, \cites{NabutovskyRotman-ShapeNets2007, Rotman-ShortGeoNet2007, ChodoshMantoulidis-PWidths2023}.
However, this outcome can be ruled out in some specific situations, as demonstrated, for example, in \cites{Cheng-StableNets2024, ChambersLiokumovichNabutovskyRotman-GeoNetsNonComp2023}.
We also mention the works \cites{staffa2023bumpy, LiokumovichStaffa-GenDen2024}, where generic non-degeneracy and density of stationary networks is showed with respect to the Riemannian metric, respectively.
Our treatment here differs in that we prescribe the structure of the minimal networks we want to construct without relying on a generic choice of the metric.
Additionally, the networks we are approximating are degenerate in the sense that they have nontrivial normal Jacobi fields.
For a survey and other problems on the matter, we refer to \cite{NabutovskyParsch-Nets2023}.

The paper is organised as follows.
In Section \ref{sec:defs} we introduce the notion of network and minimality, recall the first variation formula for the length functional, and define the two particular networks we treat in what follows.
For the ease of presentation, we first show Theorem \ref{thm:mainB} in Section \ref{sec:ball}, and in Section \ref{sec:sphere} we prove Theorem \ref{thm:mainS}.
An appendix states the Lusternik--Schnirelmann Theorem we employ to bound the number of distinct minimal networks from below.

\section*{Acknowledgement}

The author would like to thank Andrea Malchiodi and Matteo Novaga for suggesting this problem and for insightful discussions.


\section{Definitions and setup}\label{sec:defs}

Let $M^d$ be a manifold of dimension $d \geq 2$.
A \textit{network} $\G$ on $M$ is a finite set of embedded, regular curves $\g_j: [0,1] \to M$, $j = 1, \ldots, n$ such that the union of their images
\[
\im\G := \bigcup_{j=1}^n \g_j([0,1])
\]
is connected and the curves may intersect only at their endpoints.
We call the curves $\g_j$ the \textit{edges} of the network $\G$, and denote their unit tangents by $\t_j$.
For consistency, we always denote the edges of a network their unit tangents by these Greek letters, where we may drop the subindices to simplify notation whenever this is not a source of confusion.
We also call \textit{junction} a point $p \in \im\G$ such that two or more edges intersect at $p$, and we call \textit{endpoint} of the network a point $q \in \im\G$ such that it is the endpoint of one edge, but it is not a junction.

We consider two types of network, each one consisting of three edges.
The first one is the \textit{triod} on the unit ball $\B^d$ of $\R^d$, which has one junction and three endpoints constrained to be at the boundary $\p\B^d$.
In other words, if we denote its edges by $\g_j: [0,1] \to \overline{\B^d}$, $j=1,2,3$, then, after reparametisation,
\[
\g_1(0) = \g_2(0) = \g_3(0) , \qquad \g_1(1) , \g_2(1) , \g_3(1) \in \p\B^d .
\]
The second one is the $\th$-\textit{network} on the sphere $\S^d$, which has two triple junctions.
In a similar way, if we denote $\g_j: [0,1] \to \S^d$, $j=1,2,3$ its edges, then they satisfy
\[
\g_1(0) = \g_2(0) = \g_3(0) , \qquad \g_1(1) = \g_2(1) = \g_3(1).
\]

When we endow $M$ with a Riemannian metric $g$, we can define the \textit{length} of a network $\G$ to be the sum of the lengths of its edges, that is
\[
L(\G) := \sum_{j=1}^n \int_0^1 |\g_j'(t)|_{g(\g_j(t))} \;dt .
\]
We call a network $\G_0$ on $(M,g)$ \textit{stationary} if it is a critical point for the length functional, that is, for any smooth one-parameter family of variations $\G(\e) = \{ \g_j(\cdot;\e) \}_{j=1}^n$, $|\e| \ll 1$ such that $\G(0) = \G_0$ and which preserve the structure of the image $\im\G$, it holds
\[
\left. \frac{d}{d\e} \right|_{\e=0} L(\G(\e)) = 0 .
\]
The next proposition gives a formula for the first variation of the length of a network.
It is a consequence of the formula for the first variation of length of a single curve, which can be consulted, for example, in \cite{IvanovTuzhilin-MinNets1994}*{Ch. 1, Theorem 3.2}.
\begin{prp}\label{prp:1var}
Let $\G_0 = \{ \g_j \}_{j=1}^n$ be a network on a Riemannian manifold $M$.
Denote $P$ the set of its junctions, and $Q$ the set of its endpoints.
Let $\G(\e)$, $\e \in (-1,1)$, be a smooth variation of $\G_0$, and set the variational fields
\[
X_j := \left. \frac{d}{d\e} \right|_{\e=0} \g_j(\cdot;\e) .
\]
Then the first variation of the length of $\G_0$ is given by
\begin{equation}\label{eq:1var}
\sum_{p \in P} \left\langle \sum_{i=1}^{m_p} \eta_{j_i} , X(p) \right\rangle + \sum_{q \in Q} \langle \eta , X(q) \rangle - \sum_{j=1}^n \int_{\g_j} \langle X_j, \nabla_{\t_j}\t_j \rangle \;ds
\end{equation}
where $\eta_j(t) := (-1)^{t+1}\t_j(t)$, $t \in \{0,1\}$, is the exterior normal to $\g_j$ at each of its endpoints.
\end{prp}
Therefore, a network is stationary if and only if the following three conditions are satisfied:
\begin{itemize}
    \item each edge is a geodesic;
    \item at each junction the sum of the unit tangents to each edge, pointing to the \textit{exterior} of the curve, is zero;
    \item if each endpoint is constrained to the boundary of $M$, then the network is normal to $\p M$ at each endpoint.
\end{itemize}
In particular, whenever the junction is of degree two, the tangents to the two concurring curves are parallel, and the concatenation of such two curves can be parametrised as a single regular curve.
Also, since triple junctions locally minimise length \cite{IvanovTuzhilin-MinNets1994}*{Ch. 3, Theorem 2.1}, we make the following definition.

\begin{defn}[Minimal network]\label{def:min-net}
    A network on a Riemannian manifold $M$ is said to be \textit{minimal} if it is stationary, it only has triple junctions, and it is normal to the boundary of $M$ at each of its endpoints.
\end{defn}


\section{A reduction method and proof of Theorem \ref{thm:mainB}} \label{sec:ball}

We now proceed to show the existence of minimal triods on the $d$-dimensional unit ball $\B^d \subset \R^d$, $d \geq 2$, endowed with a Riemannian metric $g$ which is sufficiently close in $C^2$ to the standard Euclidean metric $\d_0$.
Consider the family of triods that have the triple junction at the origin and whose edges are straight segments whose tangents form equal angles of $2\pi/3$.
These triods are minimal in the sense of Definition \ref{def:min-net}, and can be parametrised by the \textit{Stiefel manifold} $V_2(\R^d)$ of orthonormal 2-frames in $\R^d$ in the following way: denote a point in $V_2(\R^d)$ by $e = (e_1, e_2)$, and consider the unit vectors
\begin{equation}\label{eq:u0-vecs-B}
    u_1 := e_1, \quad u_2 := -\frac{1}{2}e_1 + \frac{\sqrt{3}}{2}e_2, \quad u_3 := -\frac{1}{2}e_1 - \frac{\sqrt{3}}{2}e_2.
\end{equation}
Then the network is given simply by the curves
\[
\g_j(t) := tu_j , \quad t \in [0,1] , \quad j = 1,2,3 ,
\]
and its length is $3$ independently of the parameter $e \in V_2(\R^d)$.
Of course, such parametrisation is redundant in the sense that many different triples $u$ represent the same geometric object.
However, one can quotient out by the symmetry group of said triples, which is the permutation group $S_3$ (see Remark \ref{rem:S3action} below).
Despite this redundancy, we do not need to keep these symmetries in mind for much of the analysis done.
It is also worth noting that $V_2(\R^d)$ is homeomorphic to the unit tangent bundle of $\S^{d-1}$, and $V_2(\R^2)$ is disconnected.

For now, fix the Euclidean metric $g = \d_0$.
To simplify notation, we will denote by a subscript $0$ any object that in principle depends on the metric when evaluated at the Euclidean one, e.g., $L_0$ for $L_{\d_0}$, etc.

Let $(x,w) \in \B^d \times \S^{d-1}$.
Consider the straight line $\widehat{\g}(t; x,w) := x + tw$ emanating at $x$ with initial velocity $w$, and define $L_0 = L_0(x,w)$ to be the the first time $t$ at which $\widehat{\g}(t) \in \p\B^d$.
We can compute $L_0$ explicitly by solving $|\widehat{\g}(t)|^2 = 1$, which gives us the solution
\begin{equation}\label{eq:L0}
L_0(x,w) = - x \cdot w + \sqrt{1 + (x \cdot w)^2 - |x|^2} = - x \cdot w + \sqrt{1 - |P_w x|^2} ,
\end{equation}
where $P_w x := x - (x \cdot w)w$ denotes the orthogonal projection of $x$ onto the subspace $w^\perp \subset \R^d$.

Now, for each $(x,e) \in \B^d \times V_2(\R^d)$ we define a triod $T_0(x,e) := \{ \g_j(\cdot; x,e) \}_{j=1}^3$, whose junction is $x$ and the three endpoints are in $\p\B^d$, as
\[
\g_j(t;x,e) := x + tL_0(x,u_j)u_j , \quad t \in [0,1] , \quad j = 1,2,3 ,
\]
where the $u_j$ are defined by \eqref{eq:u0-vecs-B}.
We then define the function $F_0: \B^d \times V_2(\R^d) \to \R$ as to be the total length of the triod $T_0$, that is,
\begin{equation}\label{eq:Fg0}
F_0(x,e) := L_0(x,u_1) + L_0(x,u_2) + L_0(x,u_3) , \quad (x,e) \in \B^d \times V_2(\R^d) .
\end{equation}
Note that, as $u_1 + u_2 + u_3 = 0$, from \eqref{eq:L0} we deduce the explicit formula
\begin{equation}\label{eq:Fg00}
F_0(x,e) = \sqrt{1 - |P_1x|^2} + \sqrt{1 - |P_2x|^2} + \sqrt{1 - |P_3x|^2} ,
\end{equation}
where we let $P_j = P_{u_j}$, $j = 1,2,3$.

We now summarise some elementary properties of $F_0$ in the following lemma.

\begin{lem}\label{lem:F0}
The function $F_0$ defined above in \eqref{eq:Fg0} is smooth in $\B^d \times V_2(\R^d)$ and continuous up to $\overline{\B^d} \times V_2(\R^d)$, and satisfies
\[
\max_{(x,e) \in \B^d \times V_2(\R^d)} F_0(x,e) = 3 , \qquad \max_{(x,e) \in \p\B^d \times V_2(\R^d)} F_0(x,e) = 2 .
\]
Moreover, its only interior critical points are $(0,e)$, $e \in V_2(\R^d)$, which attain the maximum value $3$, and the critical submanifold $\{ 0 \} \times V_2(\R^d) \simeq V_2(\R^d)$ is non-degenerate in the sense of Bott, i.e. the Hessian
\[
d^2F_0(0,e)(\x,\x) = 0 \quad \text{if and only if} \quad \x \in T_e V_2(\R^d) \subset T_{(0,e)} \B^d \times V_2(\R^d) .
\]

\begin{proof}
The regularity properties are clear from formula \eqref{eq:Fg00}, and observe that $F_0$ is differentiable on $\overline{\B^d} \times V_2(\R^d)$ except at the points $(x,e)$ such that $x = P_jx$, $j = 1,2,3$.
Now, it is clear that $F_0 \leq 3$ and $F_0(0,e) = 3$ for all $e \in V_2(\R^d)$.
To see that these are the unique interior critical points, for every $x \in \B^d$ and $e \in V_2(\R^d)$ we compute directly
\begin{equation}\label{eq:dF0x}
\left. \frac{d}{ds} \right|_{s=1} F(sx,e) = - \sum_{j=1}^3 \frac{|P_jx|^2}{\sqrt{1 - |P_jx|^2}} .
\end{equation}
We see that the derivative in the above formula is equal to $0$ if and only if for every $j$ we have $P_jx = 0$, and since $\{ u_1, u_2 \}$ is linearly independent, this only occurs when $x=0$.
Thus, $dF_0(x,e) = 0$ if and only if $x=0$.
Note that \eqref{eq:dF0x} also implies that $F_0$ is strictly decreasing along the rays $s \mapsto sx$, $x \neq 0$, and so attains its minimum only at the boundary $\p\B^d \times V_2(\R^d)$.

To see that $V_2(\R^d)$ is non degenerate, first note that as every $(0,e)$ is a maximum point, we have $d^2F_0(0,e) \leq 0$ in the sense of quadratic forms on $T_{(0,e)} \B^d \times V_2(\R^d)$.
Thus, the non-degeneracy condition follows once we see that the partial second derivative $\nabla_x^2 F_0(0,e)$ is negative-definite on $T_0 \B^d \simeq \R^d$ for every $e \in V_2(\R^d)$.
We compute directly for $v \in \R^d$
\[
\nabla_x^2 F_0(0,e) (v,v) = \left. \frac{d^2}{ds^2} \right|_{s=0} F_0(sv,e) = \left. \frac{d}{ds} \right|_{s=0} - \sum_{j=1}^3 \frac{s |P_j v|^2}{\sqrt{1 - s^2|P_j v|^2}} = - \sum_{j=1}^3 |P_j v|^2 ,
\]
which again is strictly negative for every $v \neq 0$.

We now compute $\max_{(x,e) \in \p\B^d \times V_2(\R^d)} F_0(x,e)$.
Note that for $x \in \p\B^d$ formula \eqref{eq:L0} gives us
\[
F_0(x,e) = | x \cdot u_1 | + | x \cdot u_2 | + | x \cdot u_3 | .
\]
We now show that for $x \in \p\B^d$, we have
\[
F_0(x,e) = 2 \max \{ | x \cdot u_1 | , | x \cdot u_2 | , | x \cdot u_3 | \} .
\]
Let us suppose that $\max_{j \in \{1,2,3\}} | x \cdot u_j | = | x \cdot u_1 |$.
As $F_0(x,e) = F_0(-x,e)$, we can assume without loss of generality that $x \cdot u_1 \geq 0$.
Since $ x \cdot u_1 + x \cdot u_2 + x \cdot u_3 = 0$, we must have $x \cdot u_1 \neq 0$, and at least one other term in the sum is negative, say $x \cdot u_2 < 0$.
Therefore $x \cdot u_1 \geq | x \cdot u_2 | =  - x \cdot u_2$ and thus
\[
| x \cdot u_3 | = | x \cdot u_1 + x \cdot u_2 | = x \cdot u_1 + x \cdot u_2 .
\]
Putting all this together, we find
\[
F_0(x,e) = x \cdot u_1 - x \cdot u_2 + | x \cdot u_1 + x \cdot u_2 | = 2 x \cdot u_1 ,
\]
which proves our claim.
It follows that $\max_{(x,e) \in \p\B^d \times V_2(\R^d)} F_0(x,e) = 2$, and it is attained if and only if $x = \pm u_j$ for some $j \in \{ 1, 2, 3 \}$.
\end{proof}
\end{lem}

The idea now is to slightly perturb the Euclidean metric $\d_0$ on $\B^d$ so as to still have a family of triods parametrised by an open subset containing $\{ 0 \} \times V_2(\R^d)$.
Thus, we first need to show that we can define the analogous quantities for a Riemannian metric which is a perturbation of $\d_0$.

\begin{lem}\label{lem:L}
For any Riemannian metric $g$ on $\B^d$ sufficiently close to $\d_0$ in $C^2$ we have that any unit speed geodesic starting at $x \in B_{1/2}$ in direction $w \in \S^{d-1}$ hits transversely the boundary $\p\B^d$ at a first time $t = L_g(x,w)$.
Moreover, $L_g \in C^\infty(B_{1/2} \times \S^{d-1})$ and the mapping $g \mapsto L_g \in C^2(B_{1/2} \times \S^{d-1})$ is continuous in $C^2$.
In particular, $L_g \to L_0$ in $C^2$ as $g \to \d_0$ in $C^2$.

\begin{proof}
Without loss of generality we can suppose that the metric $g$ is defined in the whole of $\R^d$.
Consider the unit speed geodesic $\widehat{\g}: [0,\infty) \times B_{1/2} \times \S^{d-1} \to \R^d$ given by
\[
\widehat{\g}(t;x,w) := \exp_x(tw/|w|_{g(x)}) ,
\]
and note that it defines a smooth map which depends continuously on the metric $g$.
Define $L_g = L_g(x,w)$ as the first time $t$ such that $\widehat{\g}(t)$ hits $\p\B^d$, i.e.
\[
L_g(x,w) := \inf \{ t>0 : |\widehat{\g}(t;x,w)| \geq 1 \} .
\]
By ODE theory, if $g \to \d_0$ in $C^2$ we have that $\widehat{\g} \to \widehat{\g}_0$ in $C^2([0,T] \times B_{1/2} \times \S^{d-1}; \R^d)$ for every $T > 0$, where $\widehat{\g}_0(t;x,w) := x + tw$ is the straight line geodesic for $\d_0$.
Therefore, if $g$ is close to $\d_0$ in $C^2$, we have that $|\widehat{\g}(t;x,w)| > 1$ for all $(x,w) \in B_{1/2} \times \S^{d-1}$ if $t$ is sufficiently large, and so $L_g$ is a well defined positive number such that $\widehat{\g}(L_g;x,w) \in \p\B^d$ by continuity.
Note also that when $g = \d_0$, we have $L_g = L_0$ which is computed explicitly in \eqref{eq:L0}.

To show that $L_g = L_g(x,w)$ defines a smooth function of $(x,w)$, define
\[
S_g(t;x,w) := \frac{1}{2} (|\widehat{\g}(t;x,w)|^2 - 1)
\]
and note that, by definition, $S_g(L_g(x,w);x,w) = 0$.
By the above argument, $S_g$ belongs to $C^\infty([0,\infty) \times B_{1/2} \times \S^{d-1})$ and $g \mapsto S_g \in C^2([0,T] \times B_{1/2} \times \S^{d-1})$ is continuous in $C^2$ for every $T > 0$.
Also, for $L_0 = L_0(x,w)$ we have
\[
\frac{\p S_0}{\p t} (L_0;x,w) = (x + L_0 w) \cdot w = x \cdot w + L_0 = \sqrt{1 - |P_w x|^2} \geq \frac{\sqrt{3}}{2} ,
\]
so we can conclude by the implicit function theorem.
\end{proof}
\end{lem}

Now, let $g$ be a Riemannian metric on $\B^d$, and for each $x \in \B^d$ identify in the canonical way $T_x \B^d \simeq \R^d$.
Consider the mapping $\psi_g \colon \B^d \times V_2(\R^d) \to \B^d \times (\R^d)^3$ given by
\begin{equation}\label{eq:diffeo-B}
\psi_g(x,e) := \left( x, u_1 , - \frac{1}{2} u_1 + \frac{\sqrt{3}}{2} \n_1 , - \frac{1}{2} u_1 - \frac{\sqrt{3}}{2} \n_1 \right) , 
\end{equation}
where $(u_1, \n_1)$ is the result of applying the Gram--Schmidt process to $(e_1,e_2)$ with respect to $g(x)$.
We also denote $(x,u) = \psi_g(x,e)$, and consider the unit normal $\n_j$ to $u_j$ such that $(u_j, \n_j)$ induces the same orientation as $(e_1, e_2)$ in the plane generated by both vectors, $j = 2,3$.
For example, when $g = \d_0$, $(u_1, \n_1) = (e_1, e_2)$ is just the identity, and we recover formulas \eqref{eq:u0-vecs-B}
It is clear that $\psi_g$ is a diffeomorphism onto its image, and that it is smooth in $g$ as a parameter.

For each $(x,e) \in B_{1/2} \times V_2(\R^d)$ and a metric $g$ sufficiently close to $\d_0$ in $C^2$ we construct a triod as follows.
For every $j \in \{ 1, 2, 3 \}$ consider the unit-speed geodesic $\widehat{\g}_j: [0,\infty) \to \R^d$ given by $\widehat{\g}_j(t;x,e) := \exp_x(tu_j)$, where $(x,u) = \psi_g(x,e)$ as defined in \eqref{eq:diffeo-B}.
By Lemma \ref{lem:L} there exists a first time $t = L_g(x,u_j)$ such that $\widehat{\g}_j(t;x,e)$ hits $\p\B^d$, which is a smooth function of $(x,e)$.
Reparametrise each $\widehat{\g}_j$ to $\g_j: [0,1] \to \B^d$ as $\g_j(t) := \widehat{\g}_j(tL_g(x,u_j))$ and define the triod $T_g(x,e) := \{ \g_j(\cdot;x,e) \}_{j=1}^3$.

We can now define the \textit{reduced length functional} as the total length of the triod $T_g(x,e)$,
\begin{equation}\label{eq:Fg}
F_g(x,e) := L_g(x,u_1) + L_g(x,u_2) + L_g(x,u_3) , \quad (x,e) \in B_{1/2} \times V_2(\R^d) .
\end{equation}
It follows from Lemma \ref{lem:L} that $F_g$ is smooth and that depends continuously on the metric $g$ in the sense that $g \mapsto F_g \in C^2(B_{1/2} \times V_2(\R^d))$ is continuous in $C^2$.

Let $\x$ be a tangent vector to $B_{1/2} \times V_2(\R^d)$ at $(x,e)$.
Using formula \eqref{eq:1var} for the first variation of the length of a network, we see that
\begin{equation}\label{eq:dF}
dF_g(x,e) \cdot \x = \langle dq_1(x,e) \cdot \x, \t_1 \rangle + \langle dq_2(x,e) \cdot \x, \t_2 \rangle + \langle dq_3(x,e) \cdot \x, \t_3 \rangle ,
\end{equation}
where we denote $q_j(x,e) := \g_j(1;x,e)$.
We also note that $dq_j(x,e) \cdot \x$ is tangent to $\p\B^d$ at $q_j$ for $j = 1,2,3$.
Therefore, if $(x,e)$ is a critical point of $F_g$, from \eqref{eq:dF} we see that the orthogonality of $\t_j$ to $\p\B^d$ at $q_j$ follows after the next lemma.

\begin{lem}\label{lem:linearinv}
Let $(x,e) \in B_{1/2} \times V_2(\R^d)$.
Then if $g$ is sufficiently close to $\d_0$ in $C^2$ the linear mapping
\[
\x \in T_{(x,e)} B_{1/2} \times V_2(\R^d) \mapsto (dq_j(x,e) \cdot \x)_{j=1}^3 \in \bigoplus_{j=1}^3 T_{q_j}\p\B^d
\]
is invertible.

\begin{proof}
Write $d\psi_g(x,e) \cdot \x = (v,\eta)$, where $v \in \R^d$ and $\eta \in T_e V_2(\R^d)$.
As $g \to \d_0$ in $C^2$, by the properties of $\psi_g$, we have $\psi_g(x,e) \to \psi_0(x,e)$ and $d\psi_g(x,e) \cdot \x \to d\psi_0(x,e) \cdot \x$, and thus, by Lemma \ref{lem:L}, also $L(x,u_j) \to L_0(x,u_j)$.
Since the $q_j$ depend continuously in the metric, by approximation it is sufficient to show the lemma for $g = \d_0$.
Recall that in the Euclidean case
\[
q_j(x,e) = x + L_0(x,u_j) u_j , \quad j = 1,2,3 .
\]
To show the invertibility, we prove that the kernel of the mapping is trivial.
For this, let $\x$ be a tangent vector and suppose that
\begin{equation}\label{eq:dq=0}
dq_j(x,e) \cdot \x = v + (dL(x,u_j)\cdot(v,\eta_j)) u_j + L_0(x,u_j) \eta_j = 0 , \quad j = 1,2,3 .
\end{equation}
Let $\Pi(u)$ denote the $2$-dimensional subspace of $\R^d$ generated by $\{u_1, u_2, u_3 \}$.
We may write $v = p + q$ and $\eta_j = c_j \n_j + \z_j$ for some $c_j \in \R$ and $p \in \Pi(u)$, $q, \z_j \in \Pi(u)^\perp$.
Then equation \eqref{eq:dq=0} reads
\begin{equation}\label{eq:pq}
p = (v \cdot u_j)u_j - c_j L_0(x,u_j) \n_j ,\qquad q = - L_0(x,u_j) \z_j .
\end{equation}
In particular, the $\z_j$ are parallel vectors.
Since $u_1 + u_2 + u_3 = 0$, we also have $\eta_1 + \eta_2 + \eta_3 = 0$.
Hence, if we project onto the subspace $\Pi(u)^\perp$ we find that
\[
\z_1 + \z_2 + \z_3 = - \left( \frac{1}{L_0(x,u_1)} + \frac{1}{L_0(x,u_2)} + \frac{1}{L_0(x,u_3)} \right) q = 0 .
\]
This implies that $q=0$ and $\z_j = 0$ for every $j$.
Now, we also have that $c_1\n_2 + c_2\n_2 + c_3\n_3 = 0$, and since $\n_1 + \n_2 + \n_3 = 0$ we find that
\[
(c_1 - c_3)\n_1 + (c_2 - c_3)\n_2 = 0 .
\]
Thus by the linear independence of $\{ \n_1, \n_2\}$, we have $c_1 = c_2 = c_3$.
If we denote this common value by $c$ and plug it into the first equation in \eqref{eq:pq} we see that
\[
-cF_0(x,e) = p \cdot \n_1 + p \cdot \n_2 + p \cdot \n_3 = 0 ,
\]
and thus also $c = 0$.
This means that $\eta = 0$ and $v = p = (v \cdot u_j) u_j$ for every $j$, which in turn implies $v=0$ as $\{ u_1, u_2 \}$ is linearly independent.
Therefore, $\x = 0$ as we were to prove.
\end{proof}
\end{lem}

We can therefore conclude that the functional $F_g$ defined in \eqref{eq:Fg} is a reduced functional for the length functional in the space of triods in the sense of the following proposition.

\begin{prp}\label{prp:reduced-Fg}
Every interior critical point of $F_g$ determines a minimal triod in the sense of Definition \ref{def:min-net}.

\begin{proof}
Let $(x,e) \in B_{1/2} \times V_2(\R^d)$ be a critical point of $F_g$, and let $X \in T_{q_1} \p\B^d$ be any tangent vector.
Then by Lemma \ref{lem:linearinv} and formula \eqref{eq:dF}  we can find a tangent vector $\x \in T_{(x,e)} B_{1/2} \times V_2(\R^d)$ such that
\[
dF_g(x,e) \cdot \x = \langle X , \t_1 \rangle = 0 .
\]
By the arbitrariness of $X$ we conclude that each $\t_1$ is orthogonal to $\p\B^d$ with respect to $g$.
The same can be done for $j = 2,3$, and thus the triod $T_g(x,e)$ is minimal.
\end{proof}
\end{prp}

From Proposition \ref{prp:reduced-Fg} and Lemma \ref{lem:F0} it then follows that there exists at least one critical triod, which corresponds to a maximum point for $F_g$, since for $g$ close to $\d_0$ in $C^2$ we have
\[
\max_{(x,e) \in \p B_{1/2} \times V_2(\R^d)} F_g(x,e) < \max_{(x,e) \in B_{1/2} \times V_2(\R^d)} F_g(x,e) .
\]
This can be improved, and we make a further reduction of the variables to obtain a multiplicity result.

\begin{lem}\label{lem:Lyapunov-Fg}
For each $g$ sufficiently close to $\d_0$ in $C^2$ and $e \in V_2(\R^d)$ there exists $\e > 0$ and a unique $x = x_g(e) \in B_\e$ such that the partial derivative
\[
\nabla_x F_g(x_g(e),e) = 0 .
\]
Moreover, we have that $\| x_g \|_{C^1(V_2(\R^d))} \to 0$ as $g \to \d_0$ in $C^2$.

\begin{proof}
Define the function
\[
H_g(x,e) := \nabla_x F_g(x,e) , \quad (x,e) \in B_{1/2} \times V_2(\R^d) .
\]
As a consequence of Lemma \ref{lem:L} we see that $g \mapsto H_g \in C^1(B_{1/2} \times V_2(\R^d); \R^d)$ is continuous in $C^2$, and because of Lemma \ref{lem:F0}, we have that the partial derivative
\[
\nabla_x H_0(0,e) = \nabla_x^2 F_0(0,e)
\]
is an isomorphism for each $e \in V_2(\R^d)$.
Hence, we can apply once again the implicit function theorem to find a small enough $\e > 0$ and a unique $x = x_g(e) \in B_\e$ such that
\[
H_g(x_g(e),e) = \nabla_x F_g(x_g(e),e) = 0
\]
for every $e \in V_2(\R^d)$ and $g$ in a $C^2$-neighbourhood of $\d_0$.
\end{proof}
\end{lem}

Since the manifold $V_2(\R^d)$ is compact, after we reduce variables we can find at least two critical points, corresponding to maximum and minimum values of $F_g(x_g(e),e)$.
In order to give a more meaningful bound, we need to employ the \textit{Lusternik--Schnirelmann category} of $V_2(\R^d)$ (see Definition \ref{def:LS} in the Appendix).

\begin{prp}\label{prp:LS-B}
The Lusternik--Schnirelmann category of $V_2(\R^d)$ is equal to $4$, if $d = 2,3$, and equal to $3$ if $d \geq 4$.

\begin{proof}
If $d = 2$, we have that $V_2(\R^2)$ is homeomorphic to the disjoint union of two circles, hence $\cat(V_2(\R^2)) = 4$.
If $d = 3$, then $\cat(V_2(\R^3)) \geq 4$ (\cite{Nishimoto-LScatStiefel2007}*{Remark 1.5}), and since $\dim(V_2(\R^3)) = 3$ we have in fact an equality.
For dimensions $d \geq 4$, see \cite{Nishimoto-LScatStiefel2007}*{Theorem 1.3}.
\end{proof}
\end{prp}

We can now prove our second main theorem.

\begin{proof}[Proof of Theorem \ref{thm:mainB}]
From Lemma \ref{lem:L} and Lemma \ref{lem:Lyapunov-Fg}, for $e \in V_2(\R^d)$ and $g$ in a $C^2$-neighbourhood of $\d_0$ we can define the function $\F_g(e) := F_g(x_g(e),e)$.
By the chain rule
\[
d\F_g(e) = \nabla_e F_g(x_g(e),e) ,
\]
and hence, if $e \in V_2(\R^d)$ is a critical point of $\F_g$ then $(x_g(e),e) \in B_{1/2} \times V_2(\R^d)$ is a critical point of $F_g$, which by Proposition \ref{prp:reduced-Fg} represents a minimal triod.
Therefore, when $d \geq 3$, by the Lusternik-Schnirelmann Theorem \ref{thm:LS} and Proposition \ref{prp:LS-B} we conclude that for every $g$ in a sufficiently small $C^2$-neighbourhood of $\d_0$ either there are at least $\cat(V_2(\R^d))$ distinct critical levels for $F_g$, or else the number of critical points of $F_g$ is infinite.
In either case, we can find $\cat(V_2(\R^d))$ geometrically different minimal triods. 
When $d = 2$, since $V_2(\R^2)$ has two connected componets, we can conclude that there are at least $\cat(V_2(\R^2))/2 = 2$ geometrically different minimal triods.
\end{proof}

\begin{rem}\label{rem:S3action}
In the proof above we are not considering the symmetries present in the problem.
To account for this, one should pass to the quotient bundle
\[
\O_g := \B^d \times V_2(\R^d) / S_3
\]
where the action of $S_3$ is given fibre-wise by the permutation
\begin{equation}\label{eq:S3action}
\s (x,e) := \psi_g^{-1} \left(x, u_{\s(1)}, u_{\s(2)}, u_{\s(3)} \right), \quad \s \in S_3, \quad (x,u) = \psi_g(x,e) .
\end{equation}
Then the reduced length functional $F_g$ is invariant by this action.
It can be seen that the manifold $\O_g \simeq \B^d \times (V_2(\R^d)/S_3)$ is connected in every dimension $d \geq 2$, and independent of the metric $g$.
It can also be proved that $\cat(V_2(\R^d)/S_3) \geq \cat(V_2(\R^d))$ in dimensions $d \geq 3$, as the preimage of every contractible set in $V_2(\R^d)/S_3$ under the natural projection is also contractible in $V_2(\R^d)$.
Thus, the exact value of $\cat(V_2(\R^d)/S_3)$ would give a sharper bound on the number of geometrically distinct minimal triods in dimensions $d \geq 3$.
In dimension $d=2$, we note that with this action $V_2(\R^2)/S_3$ is diffeomorphic to $\S^1$.
\end{rem}


\section{Proof of Theorem \ref{thm:mainS}} \label{sec:sphere}

Next we show the existence of minimal $\th$-networks on the $d$-dimensional sphere $\S^d \subset \R^{d+1}$, $d \geq 2$, endowed with a Riemannian metric $g$ which is sufficiently close to the standard round metric $g_0$.
The overall idea of the proof follows closely the exposition in Section \ref{sec:ball}.

Let us consider the family of minimal $\th$-networks on $(\S^d,g_0)$ which have antipodal junctions $x,-x \in \S^d$ and edges determined by three directions forming equal $2\pi/3$ angles in $T_x \S^d$.
In a similar fashion to the case of triods, this family can be parametrised by the Stiefel manifold $V_3(\R^{d+1})$ of orthonormal 3-frames in $\R^{d+1}$ in the following way: consider $\S^d$ isometrically embedded in $\R^{d+1}$ as usual, and denote a point in $V_3(\R^{d+1})$ by $(x,e_1,e_2)$, note that $e_1, e_2 \in T_x\S^d$, and consider the three half great circles joining $x$ with $-x$ which have unit tangent vectors $u = (u_1, u_2, u_3)$ at $x$ given as in \eqref{eq:u0-vecs-B}.
Analytically, this network is given by the curves
\begin{equation}\label{eq:curves-th-g0}
\g_j(t) := (\cos \pi t)x + (\sin \pi t)u_j , \quad t \in [0,1] , \quad j = 1,2,3 .
\end{equation}
Its length is $3\pi$ independently of the parameter $(x,e) \in V_3(\R^{d+1})$.

We then perturb this family as explained next.
Let $g$ be another Riemannian metric on $\S^d$.
Fix $(x,e) \in V_3(\R^{d+1})$, and let $y \in \S^d$ be different from $x$.
We wish to join $x$ and $y$ by three curves $\g_j$ which have \textit{constant curvature} in the following sense: if $k = \nabla_\t \t$ is the curvature vector along $\g$ with respect to the Levi-Civita connection $\nabla$ induced by the metric $g$, then we require $|k|_g$ to be constant.
This can be codified by the Frenet--Serret equations
\begin{equation}\label{eq:Frenet-Serret}
    \left\{
    \begin{alignedat}{2}
        \nabla_\t \t &= k , \\
        \nabla_\t k &= -|k|^2 \t .
    \end{alignedat}
    \right.
\end{equation}
The next lemma shows that we can do this if $g$ is sufficiently close to $g_0$.
\begin{lem}\label{lem:curvature-length}
    Let $(x,v) \in T\S^d$, with $v \neq 0$.
    There exists $\e > 0$ such that for every $y \in B_\e(-x)$ and every metric $g$ sufficiently close to $g_0$ in $C^2$, there is a unique curvature vector $k_g = k_g(x,v;y) \in T_x \S^d$, a unique length $L_g(x,v;y)$ and a unique curve $\g$ parametrised by arclength which solves Frenet--Serret equations \eqref{eq:Frenet-Serret} with initial conditions
    \[
        \g(0) = x, \quad \g'(0) = v/|v|_g , \quad k(0) = k_g ,
    \]
    and such that $\g(L_g(x,v;y)) = y$.
    Moreover, $k_g$ and $L_g$ are smooth and depend continuously in $g$.
    In particular, $k_g(x,v;-x) \to 0$, $L_g(x,v;-x) \to \pi$ in $C^2$ as $g \to g_0$ in $C^2$.

    \begin{proof}
        By ODE theory, solutions to \eqref{eq:Frenet-Serret} depend continuously on the metric $g$, so it is sufficient to prove the statement for $g = g_0$.
        Consider then $\S^d$ isometrically embedded in $\R^{d+1}$ as usual, and let $\g$ be a curve in $\S^d$ parametrised by arclength $s \in \R$.
        Then, equation \eqref{eq:Frenet-Serret} can be written as a second order system in $\R^{d+1}$ as
        \begin{equation}\label{eq:Frenet-Serret-S0}
            \left\{
            \begin{alignedat}{2}
                \g''(s) + \g(s) &= k(s) , \\
                k'(s) &= -|k|^2 \g'(s) .
            \end{alignedat}
            \right.
        \end{equation}
        Fix $(x,v) \in T\S^d$ such that $|v| = 1$, and let $\g(s;k_0)$ denote the solution to the system \eqref{eq:Frenet-Serret-S0} with initial conditions
        \begin{equation}\label{eq:Frenet-Serret-S0-indata}
            \g(0) = x , \quad \g'(0) = v , \quad k(0) = k_0 ,
        \end{equation}
        where we are assuming $k_0 \in T_x\S^d$, and $\langle v , k_0 \rangle = 0$.
        Also, note that $\g(\cdot;0)$ is a geodesic, and thus it satisfies
        \begin{equation}\label{eq:geo-data}
            \g(\pi;0) = -x , \quad \g'(\pi;0) = -v
        \end{equation}
        If we consider $\g$ as a function $T_x\S^d \to \S^d$ by parametrising $T_x\S^d$ as $(s,k_0) \mapsto sv + k_0$, then we need to show that the differential $d\g(s;k_0)$ at $(s;k_0) = (\pi;0)$ is an isomorphism $T_x\S^d \to T_{-x}\S^d$.
        For this purpose, let $\x \in T_x\S^d$ be an arbitrary tangent vector.
        If we call $X(s) := \nabla_{k_0} \g(s;0) \cdot \x$, then
        \begin{equation}\label{eq:Frenet-Serret-dg}
            d\g(\pi;0) = X(\pi) + \langle \x , v \rangle \g'(\pi;0) .
        \end{equation}
        To compute $X$, we differentiate \eqref{eq:Frenet-Serret-S0} to see
        \[
            \left\{
            \begin{alignedat}{2}
                X'' + X &= \nabla_{k_0} k , \\
                (\nabla_{k_0}k)' &= 0 .
            \end{alignedat}
            \right.
        \]
        To obtain a solution, we differentiate the initial conditions \eqref{eq:Frenet-Serret-S0-indata}, so
        \[
        X(0) = 0 , \quad X'(0) = 0 , \quad (\nabla_{k_0}k)(0) = \x - \langle \x , v \rangle v .
        \]
        This implies $X(s) = (1-\cos s)(\x - \langle \x , v \rangle v)$, and so $X(\pi) = 2(\x - \langle \x , v \rangle v)$.
        After we plug this, together with \eqref{eq:geo-data}, into \eqref{eq:Frenet-Serret-dg}, we see
        \[
        d\g(\pi;0) \cdot \x = 2\x - 3\langle \x , v \rangle v ,
        \]
        which is indeed an isomorphism.

        Therefore, the lemma follows after we apply the implicit function theorem.
    \end{proof}
\end{lem}

Consider the mapping $\f_g \colon V_3(\R^{d+1}) \to \S^d \times (\R^{d+1})^3$ defined, analogously to $\psi_g$, by
\begin{equation}\label{eq:diffeo-S}
\f_g(x,e_1,e_2) := \left( x, u_1 , - \frac{1}{2} u_1 + \frac{\sqrt{3}}{2} \n_1 , - \frac{1}{2} u_1 - \frac{\sqrt{3}}{2} \n_1 \right) , 
\end{equation}
where $(u_1, \n_1)$ is the result of applying the Gram--Schmidt process to $(e_1,e_2)$.
Again, it is clear that $\f_g$ is a diffeomorphism onto its image, and that it is smooth in $g$ as a parameter.

We construct a $\th$-network as follows.
Let $g$ be close enough to $g_0$ in $C^2$ so that Lemma \ref{lem:curvature-length} holds, and, to simplify notation, let
\[
\O := \{ (x,e;y) \in V_3(\R^{d+1}) \times \S^d : y \in B_\e(-x) \} .
\]
For each $(x,e;y) \in \O$, let $\widehat{\g}_j$ be the curves given by the initial data
\[
\widehat{\g}_j(0) = x, \quad \widehat{\g}_j'(0) = u_j ,
\]
with $(x,u) = \f_g(x,e)$ as defined in \eqref{eq:diffeo-S}.
Reparametrise $\widehat{\g}_j$ to the unit interval $[0,1]$ as $\g_j(t) := \widehat{\g}_j(tL_g(x,u_j;y))$, and define the $\th$-network $\Th(x,e;y) := \{ \g_j(\cdot; x,e;y) \}_{j=1}^3$

We can now define then the \textit{reduced length functional} as the length of the network $\Th_g(x,e;y)$, in other words,
\begin{equation}\label{eq:Eg}
    E_g(x,e;y) := L_g(x,u_1;y) + L_g(x,u_2;y) + L_g(x,u_3;y) , \quad (x,e;y) \in \O ,
\end{equation}
where, as before, $(x,u) = \f_g(x,e)$.
By Lemma \ref{lem:curvature-length} and the properties of $\f_g$, $E_g$ is smooth and $g \mapsto E_g \in C^2(\O)$ is continuous in $C^2$.

Observe that, by construction, the network $\Th_g(x,e;y)$ satisfies the balancing condition on the unit tangents to the curves at the triple junction $x$, but may fail to satify it at $y$.
Hence, by the first variation formula \eqref{eq:1var} we have that the differential of $E_g$ acts on tangent vectors $(\x,w) \in \G(T(V_3(\R^{d+1}) \times \S^d))$ as
\begin{equation}\label{eq:dEg}
dE_g(x,e;y) \cdot (\x,w) = \left\langle \sum_{j=1}^3 \t_j , w \right\rangle - \sum_{j=1}^3 \int_{\g_j} \langle X_j , k_j \rangle \;ds ,
\end{equation}
where we are denoting by $k_j$ the curvature vector of $\g_j(\cdot; x,e;y)$, and by $X_j := \nabla_{x,e;y} \g_j$ the corresponding variational vector field.
Thus, in order to show that $\Th_g(x,e;y)$ is a minimal network in the sense of Definition \ref{def:min-net} we need to verify that its three curvature vectors $k_j \equiv 0$, $j = 1,2,3$, and that at the triple junction $y$ the sum of the unit tangents $\t_1 + \t_2 + \t_3 = 0$.

To do so, we first rewrite \eqref{eq:dEg} in a more convenient way.
Fix an orthonormal basis $\{ z_1, \ldots, z_{d-2} \}$ for the orthogonal complement, with respect to the metric $g$, of the plane $\Pi(x,u) \subset T_x \S^d$ spanned by $\{ u_1, u_2, u_3 \}$.
Also, recall that $\n_j$ is the unit normal to $u_j$ which preserves the orientation of $\Pi(x,u)$ induced by $\{ u_1, u_2 \}$.
Thus, $\{ \n_j, z_1, \ldots, z_{d-2} \}$ is an orthonormal basis for $u_j^\perp$, and so we can express the initial curvature vector $k_g(x,u_j,y)$ in this basis as
\begin{equation}\label{eq:kg-z-basis}
    k_g(x,u_j;y) = \k_{j0}\n_j + \sum_{i=1}^{d-2} \k_{ji}z_i .
\end{equation}
Now, transport each $\n_j, z_1, \ldots, z_{d-2}$ along $\g_j$ to satisfy the differential equations
\begin{equation}\label{eq:mod-Frenet-Serret}
    \left\{
    \begin{alignedat}{2}
        \nabla_{\t_j} \n_j &= - \k_{j0} \t_j , \\
        \nabla_{\t_j} z_i &= - \k_{ji} \t_j , \quad i = 1, \ldots, d-2 .
    \end{alignedat}
    \right.
\end{equation}
Therefore, for every time $t \in [0,1]$, we have an othonomal basis for $\t_j^\perp \subset T_{\g_j(t)} \S^d$, and such that $\langle k_j , \n_j \rangle \equiv \k_{j0}$, $\langle k_j , z_i \rangle \equiv \k_{ji}$.
Indeed, differentiating along $\g_j$
\[
\frac{d}{dt} \langle k_j , \n_j \rangle = \langle -|k_j|^2\t_j , \n_j \rangle + \langle k_j, -\k_{j0} \t_j \rangle = 0 .
\]
A similar computation verifies the fact for $\frac{d}{dt} \langle k_j , z_i \rangle = 0$.

Thus, in this basis, the curvature vector $k_j$ along $\g_j$ is expressed as
\begin{equation}\label{eq:kj-z-basis}
    k_j(t) = \k_{j0}\n_j(t) + \sum_{i=1}^{d-2} \k_{ji}z_{ji}(t) .
\end{equation}

If we let $\vp_{j0} := \langle X_j , \n_j \rangle$, $\vp_{ji} := \langle X_j, z_{ji} \rangle$, then \eqref{eq:dEg} can be written as
\begin{multline}\label{eq:dEg-phi}
    dE_g(x,e;y) \cdot (\x,w) = \left\langle \sum_{j=1}^3 \t_j , w \right\rangle + \\
    - \sum_{j=1}^3 L_g(x,u_j;y) \left( \k_{j0} \int_0^1 \vp_{j0}(t) \;dt + \sum_{i=1}^{d-2} \k_{ji} \int_0^1 \vp_{ji}(t) \;dt \right) .
\end{multline}
To simplify notation, define the vectors
\begin{equation}\label{eq:Phi-j}
    \F_j(x,e;y) := \left( \int_0^1 \vp_{j0}(t) \;dt \right)\n_j +  \sum_{i=1}^{d-2} \left( \int_0^1 \vp_{ji}(t) \;dt \right) z_i \in u_j^\perp \subset T_x\S^d ,
\end{equation}
and set $\F(x,e;y) := (\F_1,\F_2,\F_3)$.

The next lemma, which is analogous to Lemma \ref{lem:linearinv}, is needed to show that critical points of $E_g$ represent indeed minimal networks.

\begin{lem}\label{lem:linear-inv-S}
    Let $(x,e;y) \in \O$.
    Then, for $g$ sufficiently close to $g_0$ in $C^2$, the linear mapping
    \[
    \x \in T_{(x,e)} V_3(\R^{d+1}) \mapsto \F(x,e;y) \in \bigoplus_{j=1}^3 u_j^\perp
    \]
    is invertible.

    \begin{proof}
        As before, by continuity we only need to show the Lemma for $g = g_0$ and $y = -x$, and since the dimensions of the vector spaces coincide, it is sufficient to show the mapping is injective.
        
        Let $\x \in T_{(x,e)} V_3(\R^{d+1})$, and denote $d\f_0(x,e) \cdot \x = (v,\eta)$, with $\f_0$ as defined in \eqref{eq:diffeo-S}.
        Recall that each curve of the network $\Th_0(x,e;y)$ is defined by
        \[
        \g_j(t;x,e;y) = \widehat{\g}(t L_g(x,u_j;y)) ,
        \]
        where $\widehat{\g}$ is the curve given by Lemma \ref{lem:curvature-length}.
        If we set for each $j=1,2,3$
        \begin{align*}
            X_j(t) &:= \nabla_{x,e} \g_j(t;x,e;y) \cdot \x , \\
            \widehat{X}_j(s) &:= \nabla_{x,u} \widehat\g_j(s;x,u_j;y) \cdot (v,\eta_j) ,
        \end{align*}
        the variational vector fields induced by $\x$, then by the chain rule
        \[
        X'(t) = \widehat{X}'(t L_g(x,u_j;y)) + (\nabla_{x,u} L_g(x,u_j;y) \cdot (x,\eta_j)) \widehat{\g}'(t L_g(x,u_j;y)) .
        \]
        Since we only need the normal component of $X_j$ along $\g_j$, it is sufficient to compute $\widehat X_j^\perp := \widehat X_j - \langle \widehat X_j , \widehat\g_j' \rangle \widehat\g_j'$.
        Note that when when $y = -x$, the three initial curvatures given by Lemma \ref{lem:curvature-length} are zero, $\widehat\g(x,u_j;-x)$ are the geodesics given by \eqref{eq:curves-th-g0} (after reparametrisation), and $L_g(x,u_j;-x) = \pi$.
        Also, by the first variation of length \eqref{eq:1var}, we can compute $\nabla_{x,u} L_g(x,u_j;-x) \cdot (v,\eta_j) = -\langle v,u_j \rangle$.
        Therefore, as in the proof of Lemma \ref{lem:curvature-length}, we see that $\widehat X_j$ solves the Cauchy problem
        \[
            \begin{cases}
                \widehat X_j'' + \widehat X_j = \nabla_{x,e} k_j \cdot \x , \\
                \widehat X_j(0) = v , \quad \widehat X_j'(0) = \eta_j .
            \end{cases}
        \]
        Also, by the same argument as before $\nabla_{x,e} k_j \cdot \x$ is constant.
        Thus, solving for $\widehat X_j$ leads to
        \[
        \widehat X_j(s) = v\cos s + \eta_j\sin s + (1 - \cos s) \nabla_{x,e} k_j \cdot \x .
        \]
        Hence, as straightforward computation gives
        \[
        \widehat X_j(s)^\perp = (v - \langle v,u_j \rangle u_j) \cos s + (\eta_j + \langle v,u_j \rangle x) \sin s + (1 - \cos s) \nabla_{x,e} k_j \cdot \x .
        \]
        Now, when computing $X_j$ the junction $y = -x$ is fixed, and thus
        \begin{equation}\label{eq:Xj-w}
            X_j(1) = \widehat X_j(\pi) + \langle v, u_j \rangle u_j = -v + 2\nabla_{x,e} k_j \cdot \x + \langle v, u_j \rangle u_j = 0 ,
        \end{equation}
        which in turn implies
        \[
        \nabla_{x,e} k_j \cdot \x = \frac{1}{2}(v - \langle v , u_j \rangle u_j) .
        \]
        Now, since $k_0(x,u_j,-x) = 0$, by \eqref{eq:kg-z-basis} and \eqref{eq:mod-Frenet-Serret} $\n_j$, $z_{ji} \equiv z_i$ are constant as vectors in $\R^{d+1}$, orthogonal to $x$.
        Therefore, \eqref{eq:Phi-j} can be written as
        \begin{align}\label{eq:Phi0-j}
            \begin{split}
                \pi\F_j(x,e;-x)
                &= \int_0^\pi \widehat X_j(s)^\perp \;ds \\
                &= 2(\eta_j + \langle v, u_j \rangle x) + \frac{\pi}{2} \nabla_{x,e} k_j \cdot \x \\
                &= 2(\eta_j + \langle v, u_j \rangle x) + \frac{\pi}{2}(v - \langle v , u_j \rangle u_j) .
            \end{split}
        \end{align}
        
        Suppose then that each $\F_j(x,e;-x) = 0$.
        If we sum in $j$ in \eqref{eq:Phi0-j}, recalling that $u_1 + u_2 + u_3 = 0$ and $\eta_1 + \eta_2 + \eta_3 = 0$, we deduce
        \[
        3v - \sum_{j=1}^3 \langle v , u_j \rangle u_j = 0 .
        \]
        This means that $v$ belongs to the plane spanned by $\{ u_1, u_2, u_3 \}$.
        Hence, if we take the product against, say, $u_1$, and since $\langle u_1 , u_2 \rangle = \langle u_1 , u_3 \rangle = \frac{1}{2}$,
        \[
        3\langle v , u_1 \rangle = \langle v , u_1 \rangle - \frac{\langle v , u_2 \rangle + \langle v , u_2 \rangle}{2} = \frac{3}{2}\langle v , u_1 \rangle ,
        \]
        which implies $\langle v , u_1 \rangle = 0$.
        In a similar way we deduce $\langle v , u_2 \rangle = 0$, and so $v = 0$.
        Plugging this into \eqref{eq:Phi0-j} we conclude also that $\eta_j = 0$, $j=1,2,3$.
        Thus, $\x = d\f_0^{-1}(x,e) \cdot (v,\eta) = 0$, as we were to prove.
    \end{proof}
\end{lem}

We can therefore conclude that the functional $E_g$ defined in \eqref{eq:Eg} is a reduced functional for the length functional in the space of $\th$-networks in the sense of the following proposition.

\begin{prp}\label{prp:reduced-Eg}
Every critical point of $E_g$ determines a minimal $\th$-network in the sense of Definition \ref{def:min-net}.

\begin{proof}
Let $(x,e;y) \in V_3(\R^{d+1}) \times \S^d$ be a critical point of $E_g$.
Then, by Lemma \ref{lem:linear-inv-S} and formula \eqref{eq:dEg-phi}, for every $j \in \{ 1, 2, 3 \}$ and $i \in \{ 0, \ldots, d-2 \}$ we can find a tangent vector $\x \in T_{(x,e)} V_3(\R^{d+1})$ such that
\[
dE_g(x,e;y) \cdot (\x,0) = \k_{ji} = 0 .
\]
Thanks to formula \eqref{eq:kj-z-basis}, this implies that the curvature $k_j$ of each edge $\g_j$ is equal to $0$.
On the other hand, by formula \eqref{eq:dEg} this implies
\[
dE_g(x,e;y) \cdot (0,w) = \langle \t_1 + \t_2 + \t_3 , w \rangle = 0
\]
for every tangent vector $w \in T_y\S^d$.
Therefore, $\t_1 + \t_2 + \t_3 = 0$ and thus the network $\Th_g(x,e;y)$ is minimal.
\end{proof}
\end{prp}

As in Section \ref{sec:ball}, we show the non-degeneracy of the starting critical manifold of $\th$-networks for the standard metric $g_0$.
\begin{lem}\label{lem:Eg-nondeg}
    The only critical points of the function $E_0$, defined in \eqref{eq:Eg} when $g = g_0$, are $(x,e;-x)$, $(x,e) \in V_3(\R^{d+1})$, at which it attains the maximum value $3\pi$.
    Moreover, the critical manifold $\{ (x,e;y) \in V_3(\R^{d+1}) \times \S^d : x = -y \} \simeq V_3(\R^{d+1})$ is non-degenerate in the sense of Bott.

    \begin{proof}
        To compute the Hessian $d^2E_0(x,e;-x)$, we differentiate \eqref{eq:dEg} using the fact that $k_j \equiv 0$, $\t_1 + \t_2 + \t_3 = 0$, so
        \[
        d^2E_0(x,e;-x) (\x,w;\x,w) = -\sum_{j=1}^3 \int_{\g_j} \langle X_j , \nabla_{x,e;y} k_j \cdot (\x,w) \rangle \;ds
        \]
        By a similar argument as in the proof of Lemma \ref{lem:linear-inv-S}, we can use the Cauchy problem solved by $\widehat X_j$, and then equate \eqref{eq:Xj-w} to $w$ instead to $0$ to compute
        \[
        \nabla_{x,e;y} k_j \cdot (\x,w)(t) = \frac{1}{2}(v+w - \langle v+w , u_j \rangle u_j) , \quad j = 1,2,3 .
        \]
        which again is constant as a vector in $\R^{d+1}$.
        On the other hand, by formula \eqref{eq:Phi0-j} for $\F_j(x,e;-x)$ we then deduce
        \begin{align*}
            \begin{split}
                d^2E_0(x,e;-x) (\x,w;\x,w)
                &= -\sum_{j=1}^3 2\langle \eta_j , \nabla_{x,e;y} k_j \cdot (\x,w) \rangle + \pi |\nabla_{x,e;y} k_j \cdot (\x,w)|^2 \\
                &= -\frac{\pi}{4} \sum_{j=1}^3 |v+w - \langle v+w , u_j \rangle u_j|^2
            \end{split}
        \end{align*}
        Therefore, $d^2E_0(x,e;-x)$ is negative semidefinite.
        On the other hand, if $(\x,w)$ is in its kernel, then $v+w = \langle v+w , u_j \rangle u_j$ for every $j$, which can only be possible if $v = - w$ since $\{ u_1 , u_2 \}$ is linearly independent.
        The lemma follows after we note that a tangent vector satisfying the relation $v = -w$ is tangent to $V_3(\R^{d+1}) \subset V_3(\R^{d+1}) \times \S^d$.
    \end{proof}
\end{lem}

Lemma \ref{lem:Eg-nondeg} allows us to further reduce the variables as shown in next lemma, which ultimately leads to the existence and multiplicity of minimal $\th$-networks.

\begin{lem}\label{lem:Lyapunov-Eg-S}
For each $g$ sufficiently close to $g_0$ in $C^2$ and $(x,e) \in V_3(\R^{d+1})$ there exists $\e > 0$ and a unique $y = y_g(x,e) \in B_\e(-x)$ such that the partial derivative
\[
\nabla_y E_g(x,e; y_g(x,e)) = 0 .
\]
Moreover, we have that $y_g(x,e) \to -x$ in $C^1$ as $g \to g_0$ in $C^2$.

\begin{proof}
Define the function
\[
H_g(x,e;y) := \nabla_y E_g(x,e;y) , \quad (x,e;y) \in \O .
\]
As a consequence of Lemma \ref{lem:curvature-length} we see that $g \mapsto H_g \in C^1$ is continuous in $C^2$, and because of Lemma \ref{lem:Eg-nondeg}, we have that the partial derivative
\[
\nabla_y H_0(x,e;-x) = \nabla_y^2 E_0(x,e;-x)
\]
is an isomorphism for each $(x,e) \in V_3(\R^{d+1})$.
Hence, as before, we can apply the implicit function theorem to find a small enough $\e > 0$ and a unique $y_g = y_g(x,e) \in B_\e(-x)$ such that
\[
H_g(x,e; y_g(x,e)) = \nabla_y E_g(x,e; y_g(x,e)) = 0
\]
for every $(x,e) \in V_3(\R^{d+1})$ and $g$ in a $C^2$-neighbourhood of $g_0$.
\end{proof}
\end{lem}

The last ingredient needed the Lusternik--Schnirelmann category (see Definition \ref{def:LS} in the Appendix) of $V_3(\R^{d+1})$.

\begin{prp}\label{prp:LScat-S}
The Lusternik--Schnirelmann category of $V_3(\R^{d+1})$ is equal to $8$, if $d=2$, greater that or equal to $5$, if $d = 3,4$, and equal to $4$, if $d \geq 5$.

\begin{proof}
When $d = 2$, we note that $V_3(\R^3)$ is homeomorphic to the orthogonal group $O(3)$, which in turn is homeomorphic to two disjoint copies of the projective space $\RP^3$, whose category is equal to $4$ \cite{CLOT-LScat2003}*{Ex. 1.7}.
For $d = 3,4$, we cite \cite{Nishimoto-LScatStiefel2007}*{Remark 1.5} which asserts $\cat(V_3(\R^{d+1})) \geq 5$, and for $d \geq 5$, \cite{Nishimoto-LScatStiefel2007}*{Theorem 1.3} gives $\cat(V_3(\R^{d+1})) = 4$.
\end{proof}
\end{prp}

We are now set to prove Theorem \ref{thm:mainS}.

\begin{proof}[Proof of Theorem \ref{thm:mainS}]
By Lemma \ref{lem:curvature-length}, there is a $C^2$-neighbourhood of $g_0$ such that for every metric $g$ we can define the reduced length functional $E_g$ as in \eqref{eq:Eg}.
If we set
\[
\Psi_g(x,e) := E_g(x,e; y_g(x,e)) , \quad (x,e) \in V_3(\R^{d+1}) ,
\]
where $y_g$ is given in Lemma \ref{lem:Lyapunov-Eg-S}, by the chain rule each critical point of $\Psi_g$ determines a critical point of $E_g$.
By the Lusternik--Schnirelmann Theorem \ref{thm:LS}, we can then bound from below the number of geometrically distinct critical points of $E_g$ by the Lusternik--Schnirelmann category of the manifold $V_3(\R^{d+1})$, for $d \geq 3$, and by $\cat(V_3(\R^{d+1}))/2$, for $d=2$, which is given by Proposition \ref{prp:LScat-S}.
Finally, by Proposition \ref{prp:reduced-Eg} every critical point of $E_g$ represents a minimal $\th$-network in the sense of Definition \ref{def:min-net}.
\end{proof}

\begin{rem}\label{rem:S3action-S}
As pointed in Remark \ref{rem:S3action}, it could be possible to consider the quotient manifold $V_3(\R^{d+1})/S_3$, where the action is defined analogously to \eqref{eq:S3action}, to obtain a more precise bound on the number of minimal $\th$-networks in dimensions $d \geq 3$.
For the case $d = 2$, the quotient space is an $\S^1$-bundle over $\S^2$ with \textit{clutching function} $z \mapsto z^6$ (see \cite{Hatcher-Ktheory2017}*{Chapter 1}), where we are identifying $\S^1$ as the set of complex numbers of unit absolute value, whose category according to \cite{Iwase-LScatSphere-2003}*{Theorem 2.7} is $4$.
\end{rem}


\begin{appendices}

\section{Appendix}

For the reader's convenience, we define the Lusternik--Schnirelmann category and state the Lusternik--Schnirelmann Theorem, both of which are utilized in the proofs presented in the preceding sections.
For a more comprehensive treatment and applications we refer the reader to \cite{CLOT-LScat2003}.

Let $M$ be a topological space.
A subset $A \subset M$ is \textit{contractible} in $M$ if the inclusion $i: A \to M$ is homotopic to a constant.
The following is a less general definition which we adapted for the sake of brevity and clarity. (cf. \cite{AmbrosettiMalchiodi-Nonlinear2007}*{Definition 9.2}, \cite{CLOT-LScat2003}*{Definition 1.1}).

\begin{defn}[Lusternik--Schnirelmann category]\label{def:LS}
Let $A \subset M$ be as above.
The \textit{Lusternik--Schnirelmann category} of $A$ with respect to $M$, denoted by $\cat(A,M)$, is the least integer $k$ such that $A \subset A_1 \cup \cdots \cup A_k$, with $A_j$ closed contractible in $M$, $j = 1, \ldots, k$.
We set $\cat(A,M) = \infty$ if there are no integers with the above property.
We also denote $\cat(M) := \cat(M,M)$.
\end{defn}

We remark that some authors prefer to define the Lusternik--Schnirelmann category to be one less than our definition.
In this way, the category of a contractible set becomes $0$ instead of $1$.

The following theorem gives a bound on the number of critical levels and points of a smooth function (cf. \cite{AmbrosettiMalchiodi-Nonlinear2007}*{Theorem 9.10}, \cite{CLOT-LScat2003}*{Proposition 1.23}).

\begin{thm}[Lusternik--Schnirelmann]\label{thm:LS}
Let $M$ be a smooth compact manifold which is connected, and suppose that $F: M \to \R$ is a smooth function with critical point set $Z := \{ x \in M : dF(x) = 0 \}$.
Then $F$ has at least $\cat(M)$ critical points.
More precisely, let
\[
C_k := \{ A \subset M : \cat(A,M) \geq k \}
\]
and define
\[
c_k := \inf_{A \in C_k} \max_{x \in A} F(x) ,
\]
then
\begin{itemize}
    \item any $c_k$ is a critical level of $F$;
    \item if for some integer $q \geq 1$ there holds $c_k = c_{k+1} = \cdots = c_{k+q}$, and we denote by $c$ the common value, then $\cat(F^{-1}(c) \cap Z,M) \geq q+1$.
\end{itemize}
In particular, whenever two critical levels coincide, $F$ has an infinite number of critical points.
\end{thm}

\end{appendices}


\end{document}